\documentclass[11pt]{amsart}
\usepackage{mathrsfs,latexsym,amsfonts,amssymb}
\usepackage{hyperref}
\hypersetup{colorlinks=true,
            linkcolor=cyan,
            urlcolor=red,
            citecolor=green,
            }
\newtheorem{thm}{Theorem}[section]

\newtheorem{lemma}[thm]{Lemma}

\newtheorem{defn}[thm]{Definition}
\newtheorem{cor}[thm]{Corollary}
\newtheorem{ques}[thm]{Question}
\newtheorem{prop}[thm]{Proposition}

\begin{document}

	\title[The separability embedding of $\sigma$-compact strongly topological gyrogroups]
	{The separability embedding of $\sigma$-compact strongly topological gyrogroups}
	
	\author{shumin Lai}
	\address{(shumin Lai): School of mathematics and statistics, Minnan Normal University, Zhangzhou 363000, P. R. China}
	\email{shuminlai2026@163.com}
	
	\author{Fucai Lin*}
	\address{(Fucai Lin): 1. School of mathematics and statistics, Minnan Normal University, Zhangzhou 363000, P. R. China; 2. Fujian Key Laboratory of Granular Computing and Application, Minnan Normal University, Zhangzhou 363000, P. R. China}
	\email{linfucai2008@aliyun.com; linfucai@mnnu.edu.cn}

\thanks{The authors are supported by Fujian Provincial Natural Science Foundation of China (No: 2024J02022) and the NSFC (No. 11571158).}
\thanks{* Corresponding author}
	
	
\begin{abstract}
In this paper, it is shown that every right $\omega$-narrow strongly topological gyrogroup $G$ is right $\omega$-balanced by applying the gyrosemidirect product groups. Then we investigate the class of $\sigma$-compact strongly topological gyrogroups, and conclude that every $\sigma$-compact strongly topological gyrogroup is range-metrizable. By applying these results, we discuss the separability embedding of $\sigma$-compact strongly topological gyrogroups, and claim that the following three statements (a)-(c) are equivalent for any $\sigma$-compact strongly topological gyrogroup $G$:

\smallskip
(a) $G$ is homeomorphic to a subspace of a separable regular space;

\smallskip
(b) $G$ is topologically gyrogroup isomorphic to a subgyrogroup of a separable strongly topological gyrogroup;

\smallskip
(c) $G$ is topologically gyrogroup isomorphic to a closed subgyrogroup of a separable path-connected, locally path-connected strongly topological gyrogroup.

The above results extend the classical results from topological groups to the class of strongly topological gyrogroups in the literature.		
\end{abstract}
	
	\maketitle	
\section{Introduction and Preliminaries}
In this paper, all topological spaces are assumed to be Hausdorff. Moreover, the sets of the first infinite ordinal and positive integers are denoted by $\omega$ and $\mathbb{N}$ respectively. The weight $w(X)$ of a topological space $X$ is defined as the smallest cardinal number of the form $|\mathscr{B}|$, where $\mathscr{B}$ is a base of topology in $X$. Readers may refer to \cite{AT, key9, 1} for terminology and notations not
explicitly given here.

Gyrogroup, an algebraic structure, primarily originates from A. Ungar's research on Einstein's velocity addition and Thomas precession. Gyrogroup as a generalized form of group, inherits the core elements of a group, namely closure, identity element, and inverse element, and does not explicitly have associativity. In \cite{10}, W. Atiponrat introduced the concept of topological gyrogroups and extended relevant results from topological groups to topological gyrogroups by overcoming the lack of associativity.
A. Ungar investigated numerous properties of gyrogroups, such as the gyrotranslation theorem, the gyrosemidirect product group, etc. in \cite{1}, which serves as a powerful tool for the study of topological gyrogroup. In \cite{11}, J. Wattanapan et al. verified that every strongly topological gyrogroup can be embedded as a closed subgyrogroup of a path-connected and locally path-connected topological gyrogroup, then Y. Jin and L. Xie in \cite{7} extended the above result in topological gyrogroups.
A. Leiderman et al. in \cite{12} gave a complete characterization of subgroups of separable topological groups, and proved the following theorem.

\begin{thm}\label{t-3333}
Let $G$ be an $w$-narrow topological group $G$. Then the following three statements are equivalent:
\begin{enumerate}
\item $G$ is homeomorphic to a subspace of a separable regular space;
			
\item $G$ is topologically isomorphic to a subgroup of a separable topological group;
			
\item  $G$ is topologically isomorphic to a closed subgroup of a separable path-connected, locally path-connected topological group.
\end{enumerate}
\end{thm}

Therefore, it is natural to pose the following question.

\begin{ques}\label{qq}
Whether can we generalize the class of topological groups in Theorem~\ref{t-3333} to the class of (strongly) topological gyrogroups?
\end{ques}

In this paper, we mainly consider Question~\ref{qq}, and give some partial answers to it. Next, we give some terminology and notations.

\begin{defn}[\cite{10}]\label{d-1.1}
Let $G$ be a nonempty set and let $\oplus:G \times G \rightarrow G$ be a binary operation on $G$. Then the pair $(G, \oplus)$ is called a {\color{blue} groupoid}. A function $f$ from a groupoid $(G_1, \oplus_1)$ to a groupoid $(G_2, \oplus _2)$ is said to be a {\color{blue} groupoid homomorphism} if $f(x \oplus _1 y) = f(x) \oplus _2 f(y)$ for any elements $x, y \in G_1$.  In addition, a bijection groupoid homomorphism from a groupoid $(G, \oplus)$ to itself is called a {\color{blue} groupoid automorphism}. We write $Aut(G, \oplus)$ for the set of all automorphisms of a groupoid $(G, \oplus)$.
\end{defn}

\begin{defn}[\cite{10}]\label{d-1.2}
Let $(G, \oplus)$ be a groupoid. The system $(G, \oplus)$ is called a {\color{blue} gyrogroup}, if its binary operation satisfies the following conditions:
\begin{itemize}
\item[(G1)] There exists a unique identity element $0 \in G$ such that $0 \oplus a = a = a \oplus 0$ for all $a \in G$.

\item[(G2)] For each $x \in G$, there exists a unique inverse element $\ominus x \in G$ such that $\ominus x \oplus x = 0 = x \oplus (\ominus x)$.

\item[(G3)] For all $x, y \in G$, there exists $\mathrm{gyr}[x,y] \in Aut(G, \oplus)$ with the property that $$x \oplus (y \oplus z) = (x \oplus y) \oplus \mathrm{gyr}[x,y](z)$$ for all $z \in G.$

\item[(G4)] For any $x, y \in G$, $\mathrm{gyr}[x \oplus y, y] = \mathrm{gyr}[x,y]$.
\end{itemize}
\end{defn}

\begin{defn}[\cite{10}]\label{d-1.3}
Let $(G, \oplus)$ be a gyrogroup. A nonempty subset $H$ of $G$ is called a {\color{blue} subgyrogroup}, denoted by $H \leq G$, if the following statements hold:

\smallskip
\indent\textup{(i)} The restriction $\oplus|_{H \times H}$ is a binary operation on $H$, i.e. $(H, \oplus|_{H \times H})$ is a groupoid.

\smallskip
\indent\textup{(ii)} For any $x, y \in H$, the restriction of $\mathrm{gyr}[x,y]$ to $H$, $\mathrm{gyr}[x,y]|_H:H \rightarrow \mathrm{gyr}[x,y](H)$, is a bijective homomorphism.

\smallskip
\indent\textup{(iii)} $(H, \oplus|_{H \times H})$ is a gyrogroup.

\smallskip
\indent Furthermore, a subgyrogroup $H$ of $G$ is said to be an {\color{blue} $L$-subgyrogroup} \cite{6}, denoted by $H \leq_L G$, if $\mathrm{gyr}[a,h](H) = H$ for all $a \in G$ and $h \in H$.
\end{defn}

\begin{defn}[\cite{6}]\label{d-1.4}
The gyrogroup cooperation $\boxplus$ is defined by the equation
$$x \boxplus y = x \oplus \mathrm{gyr}[x, \ominus y](y), x, y \in G.$$
\end{defn}

\begin{thm}[\cite{1}]\label{t-1.1}
Let $(G, \oplus)$ be a gyrogroup. Then, for any $x,y,z \in G$, the following equations hold.
\smallskip
\\\indent\textup{(1)} $(\ominus x) \oplus (x \oplus y) = y$.
\\\indent\textup{(2)} $(x \oplus (\ominus y)) \oplus \mathrm{gyr}[x, \ominus y](y) = x$.
\smallskip
\\\indent\textup{(3)} $(y \ominus x) \boxplus x = y$.
\smallskip
\\\indent\textup{(4)} $(y \boxminus x) \oplus x = y$.
\smallskip
\\\indent\textup{(5)} $\mathrm{gyr}[x,y](z) = \ominus(x \oplus y) \oplus (x \oplus (y \oplus z))$.
\end{thm}

\begin{defn}[\cite{10}]\label{d-1.5}
A triple $(G, \tau, \oplus)$ is called a {\color{blue} topological gyrogroup} if the following statements hold:

\smallskip
\indent\textup{(1)} $(G, \tau)$ is a topological space.

\smallskip
\indent\textup{(2)} $(G, \oplus )$ is a gyrogroup.

\smallskip
\indent\textup{(3)} The binary operation $\oplus: G \times G \rightarrow G$ is jointly continuous when $G \times G$ is endowed with the product topology, and the operation of taking the inverse $\ominus(\cdot): G \rightarrow G$, i.e. $x \rightarrow \ominus x$, is also continuous.
\end{defn}

\begin{defn}(\cite{BL}) Let $G$ be a topological gyrogroup. We say that $G$ is a {\color{blue} strongly topological gyrogroup} if there exists a neighborhood base $\mu$ of 0 such that for all $U \in \mu$, $gyr[x, y](U) = U$ for any $x, y \in G$. For convenience, we say that $G$ is a {\it strongly topological gyrogroup with a neighborhood base $\mu$ of 0}. Moreover, we may assume that each element of $\mu$ is symmetric.
\end{defn}	

 A well-known example of a strongly topological gyrogroup, which is not a topological group, is $M\ddot{o}bius$ topological gyrogroup, see \cite{BL}.

\begin{defn}[\cite{13}]\label{d-1.6}
A topological gyrogroup $G$ is called {\color{blue} left (right) $w$-narrow}, if for every open neighborhood $V$ of the identity element $0$ in $G$, there exists a countable subset $A$ of $G$ such that $G = A \oplus V (G = V \oplus A)$. If $G$ is left $w$-narrow and right $w$-narrow, then $G$ is {\color{blue} $w$-narrow}.
\end{defn}

\begin{defn}[\cite{2}]\label{d-1.7}
For a topological gyrogroup $G$ and $A \subseteq G$, the set $A$ is said to be {\color{blue} left (right) invariant} if $(x \oplus A) \oplus (\ominus x) = A$ $((x \oplus (A \oplus (\ominus x))) = A)$, for each $x \in G$. The set $A$ is {\color{blue} invariant} if it is not only left invariant but also right invariant.

The {\color{blue} left (right) invariance number} of $G$ is countable, denoted by $linv(G) \leq \omega$ ($rinv(G) \leq \omega$), if for each open neighborhood $U$ of $0$, we can find a countable family $\gamma$ of open neighborhoods of $0$ such that for each $x \in G$, $(x \oplus V) \oplus (\ominus x) \subseteq U$ ($x \oplus (V \oplus (\ominus x)) \subseteq U$) for some $V \in \gamma$. We say that $G$ has {\color{blue} countable invariance number} if $linv(G) \leq \omega$ and $rinv(G) \leq \omega$ both hold. The topological gyrogroup $G$ with $linv(G) \leq \omega$ ($rinv(G) \leq \omega$) is also called  {\color{blue} left $\omega$-balanced (right $\omega$-balanced)}. We call $G$  {\color{blue} $\omega$-balanced} if it is left $\omega$-balanced and right $\omega$-balanced.
\end{defn}

\begin{defn}[\cite{1}]\label{d-1.8}
Let $G = (G, \oplus)$ be a gyrogroup, and let $Aut(G) = Aut(G, \oplus)$ be the \textit{automorphism group} of $G$. A \textit{gyroautomorphism group}, $Aut_0(G)$, of $G$ is any subgroup of $Aut(G)$ containing all the gyroautomorphisms $\mathrm{gyr}[a,b]$ of $G$, $a, b \in G$.

The \textit{gyrosemidirect product group}
$G \times Aut_0(G)$
of a gyrogroup $G$ and any gyroautomorphism group, $Aut_0(G)$, is a group of pairs $(x, X)$, where $x \in G$ and $X \in Aut_0(G)$, with operation given by the  {\color{blue} gyrosemidirect product}
$(x, X)(y, Y) = (x \oplus Xy, \mathrm{gyr}[x, Xy]XY)$.
In analogy with the notion of the semidirect product in group theory, the gyrosemidirect product group
$G \times Aut(G)$ is called the  {\color{blue} gyroholomorph} of $G$.
\end{defn}

\begin{defn}\cite{BL}
Let $G$ be a gyrogroup, with the identity element $e$, and let $N$ be a real-valued function on $G$. Then $N$ is called a {\color{blue} prenorm} on $G$ if the following conditions hold for all $x, y\in G$:

\smallskip
\indent\textup{(PN1)} $N(e)=0$;

\smallskip
\indent\textup{(PN2)} $N(X\oplus y)\leq N(x)+N(y)$;

\smallskip
\indent\textup{(PN3)} $N(\ominus x)=N(x)$.
\end{defn}

\section{Some lemmas}
In this section, we mainly give some technical lemmas in order to prove our main theorems.

First, we consider the following question which was posed by M. Bao in \cite[Question 4.12]{2}. Indeed, this question appears naturally since every $w$-narrow topological group is $w$-balanced.

{\bf Question \cite[Question 4.12]{2}:} Must every $w$-narrow topological gyrogroup $G$ be $w$-balanced? Additionally, does this property hold if $G$ is assumed to be a strongly topological gyrogroup?

The first theorem, which is proved by by using the concept of gyrosemidirect product groups, gives a partial answer to the above question.

\begin{thm}\label{t-2}
Every right $\omega$-narrow strongly topological gyrogroup $G$ is right $\omega$-balanced.
\end{thm}

\begin{proof}
Suppose that $G$ is a right $\omega$-narrow strongly topological gyrogroup with a symmetric neighborhood base $\mathscr{U}$ at $0$ satisfying $\mathrm{gyr}[x, y](U) = U$ for each $U \in \mathscr{U}$ and any $x,  y \in G$. To verify that $G$ is right $\omega$-balanced, it suffices to show that for each $U \in \mathscr{U}$, there exists a countable family $\gamma$ of open neighborhoods of $0$ such that for each $x \in G$, there exists $O \in \gamma$ satisfying $x \oplus (O \oplus (\ominus x) )\subseteq U$.

Take any $U \in \mathscr{U}$. Then there exists $V \in \mathscr{U}$ such that $(V \oplus V )\oplus V \subseteq U$.
Since $G$ is right $\omega$-narrow, there exists a countable subset $A = \{a_n : n \in \mathbb{N}\} \subseteq G$ such that
$$G = V \oplus A = \bigcup_{n \in \mathbb{N}} (V \oplus a_n).$$
For each $a_n \in A$, there exists $U_n \in \mathscr{U}$ such that
$a_n \oplus (U_n \oplus (\ominus a_n)) \subseteq V.$
Put $\gamma = \{U_n : n \in \mathbb{N}\}$. Clearly, $\gamma$ is a countable family. We claim that for each $x \in G$, there exists $n\in\mathbb{N}$ such that $x \oplus (U_{n} \oplus (\ominus x) )\subseteq U$.

Indeed, fix any $x \in G$. Then there exist $n \in \mathbb{N}$ and $v \in V$ such that $x = v \oplus a_n$.
Put $\Phi=x \oplus ( U_n \oplus (\ominus x))$. We shall prove that $\Phi\subset U$.
Let $P=G \times Aut_0(G)$ be the gyrosemidirect product group. Clearly, we have
$$(v, I)(a_n, I) = (v \oplus a_n, \mathrm{gyr}[v, a_n]) = (v \oplus a_n, I) \cdot (0, \mathrm{gyr}[v, a_n]),$$
hence it follows that $$(x, I) = (v \oplus a_n, I) = (v, I)(a_n, I)(0, \mathrm{gyr}[v, a_n])^{-1}.$$
Then $(x, I)^{-1} = (0, \mathrm{gyr}[v, a_n]) (a_n, I)^{-1} (v, I)^{-1}$.

Put $\hat{\Phi}=(x, I) \cdot \widehat{U_n} \cdot (x, I)^{-1}$ in $P$, where $\widehat{U_n} = \{ (w, I) \mid w \in U_n \}$. Indeed, $$\hat{\Phi} = \left\{ \Big( x \oplus (w \oplus (\ominus x)), \quad gyr[x \oplus w, gyr[x, w](\ominus x)] \circ gyr[x, w] \Big) \mathrel{\bigg|} w \in U_n \right\}.$$
Then
$$\hat{\Phi} = \underbrace{(v, I)(a_n, I)(0, X)^{-1}}_{\hat{x}} \cdot \widehat{U_n} \cdot \underbrace{(0, X)(a_n, I)^{-1}(v, I)^{-1}}_{\hat{x}^{-1}},$$ where $X = \mathrm{gyr}[v, a_n]$.
For any $w\in U_{n}$, by using the associativity of the gyrosemidirect product group $P$, we conclude that
\begin{align*}
	    		(0, X)^{-1} (w, I) (0, X)&=(0, X^{-1}) (w, I) (0, X)\\
                &=(X^{-1}(w), X^{-1}) (0, X)\\
                &=(X^{-1}(w), X^{-1}X)\\
	    		&=(X^{-1}(w), I).
	    	\end{align*}
Since $G$ is a strongly topological gyrogroup and $U_n \in \mathscr{U}$, we have $X^{-1}(U_n) = U_n$, hence
$$(0, X)^{-1} \cdot \widehat{U_n} \cdot (0, X) = \{ (X^{-1}(w), I) \mid w \in U_n \} = \{ (w, I) \mid w \in U_n \} = \widehat{U_n}.$$
Now we have $$ (a_n, I) (w, I) (a_n, I)^{-1} = (a_n \oplus w, \mathrm{gyr}[a_n, w])\cdot (\ominus a_n, I),$$
where $w \in U_n$. Put $u=a_n\oplus w$ and $Y=\mathrm{gyr}[a_n, w]$.
Then $$(a_n, I) (w, I) (a_n, I)^{-1} = (u, Y) (\ominus a_n, I)= (u \oplus Y(\ominus a_n), \quad \mathrm{gyr}[u, Y(\ominus a_n)] \circ Y \circ I).$$
Therefore, we get an ordered pair $(s, K)$, where
$$s = (a_n \oplus w) \oplus \mathrm{gyr}[a_n, w](\ominus a_n)=a_n \oplus (w \oplus (\ominus a_n)) \in V$$ and $K = \mathrm{gyr}[u, Y(\ominus a_n)] \circ Y$.
Therefore, we have that $$ (v, I) (s, K) (v, I)^{-1} = (v, I) (s, K) (\ominus v, I) = (v \oplus s, \mathrm{gyr}[v, s] \circ K) (\ominus v, I).$$
Put $Z=\mathrm{gyr}[v, s] \circ K$.
Then  $$(v, I) (s, K) (v, I)^{-1} = (v \oplus s, Z) (\ominus v, I) = ( (v \oplus s) \oplus Z(\ominus v), \mathrm{gyr}[v \oplus s, Z(\ominus v)]Z).$$
Now it is easy to see that
\begin{equation}
(v \oplus s) \oplus Z(\ominus v) \in (V \oplus V) \oplus V \subseteq U.
\tag{$*$}\label{eq:star}
\end{equation}
From ($*$), it follows that for each $w \in U_n$, we have $x \oplus (w \oplus (\ominus x))\in (V \oplus V) \oplus V \subseteq U$. Hence we have $x\oplus (U_n \oplus (\ominus x) )\subseteq U$.
Therefore, $G$ is right $\omega$-balanced. The proof is completed.
\end{proof}

However, the following question is still unknown for us.

\begin{ques}
Is every left $\omega$-narrow strongly topological gyrogroup $G$ left $\omega$-balanced?
\end{ques}

The following lemma was proved in \cite{BL}, which plays an important role in this paper.

\begin{lemma} \cite[Lemma 3.12]{BL}\label{t-3}
Let $G$ be a strongly topological gyrogroup with the symmetric neighborhood base $\mathscr{U}$ at $0$, and let $\{U_n : n \in \omega\}$ and $\{V(m/2^n) : n, m \in \omega\}$ be two sequences of open neighborhoods satisfying the following conditions (1)-(5):
\begin{enumerate}
    \item[(1)] $U_n \in \mathscr{U}$ for each $n \in \omega$.
    \item[(2)] $U_{n+1} \oplus U_{n+1} \subseteq U_n$, for each $n \in \omega$.
    \item[(3)] $V(1) = U_0$;
    \item[(4)] For any $n \ge 1$, put
    $$ V(1/2^n) = U_n, \ V(2m/2^n) = V(m/2^{n-1}) $$
    for $m = 1, \dots, 2^{n}$, and
    $$ V((2m + 1)/2^n) = U_n \oplus V(m/2^{n-1}) = V(1/2^n) \oplus V(m/2^{n-1}) $$
    for each $m = 1, \dots, 2^{n} - 1$;
    \item[(5)] $V(m/2^n) = G$ when $m > 2^n$.
\end{enumerate}
Then there exists a prenorm $N$ on $G$ satisfies the following conditions:
\begin{enumerate}
    \item[(a)] for any fixed $x, y \in G$, we have $N(\mathrm{gyr}[x, y](z)) = N(z)$ for any $z \in G$;
    \item[(b)] for any $n \in \omega$,
    $$ \{x \in G : N(x) < 1/2^n\} \subseteq  U_n \subseteq \{x \in G : N(x) \le 2/2^n\}. $$
\end{enumerate}
\end{lemma}

\begin{lemma}\label{t-4}
Let $G$ be a right $\omega$-balanced topological gyrogroup, and let $\gamma$ be a countable family of open neighborhoods of the identity element $0 \in G$. Then there exists a countable family $\gamma^*$ of open neighborhoods of $0$ with the following properties:
\begin{enumerate}
\item[(1)] $\gamma \subseteq \gamma^*$;

\item[(2)] the intersection of any finite subfamily of $\gamma^*$ belongs to $\gamma^*$;

\item[(3)] for each $U \in \gamma^*$, there exists $V \in \gamma^*$ such that $V = \ominus V$ and $V \oplus V \subseteq U$;

\item[(4)] for each $U \in \gamma^*$ and $a \in G$, there exists $V \in \gamma^*$ such that $ a \oplus (V \oplus(\ominus a)) \subseteq U$.
\end{enumerate}
\end{lemma}
\begin{proof}
For each $U \in \gamma$, fix an open neighborhood  $V_U$ of $0$ such that $V_U = \ominus V_U$ and $V_U \oplus V_U \subseteq U$.
Since $G$ is right $\omega$-balanced, we can find a countable family $\mathcal{V}_U$ of open neighborhoods of $0$ subordinated to $U$, such that for each $x \in G$ there exists $W_U \in \mathcal{V}_U$ with $ x \oplus (W_U \oplus (\ominus x)) \subseteq U$.

Now let $$ \phi(\gamma) = \left\{ \bigcap \lambda : \lambda \subseteq \gamma, |\lambda| < \omega \right\} \cup \bigcup \{ \mathcal{V}_U : U \in \gamma \} \cup \{ V_U : U \in \gamma \}.$$

We put $\gamma_0 = \gamma$, $\gamma_1 = \phi(\gamma_0)$, and repeat this operation, defining by induction, for each $n \in \mathbb{N}$, we have $\gamma_{n+1} = \phi(\gamma_n)$. Let $\gamma^* = \bigcup_{n \in \mathbb{N}} \gamma_n$; then $\gamma^*$ satisfies conditions (1)-(4).
\end{proof}

\begin{lemma} \label{t-5}
Let $G$ be a right $\omega$-balanced topological gyrogroup, and $U$ be an open neighborhood of the identity element $0$ in $G$. Then there exists a sequence $\{U_n : n \in \omega \}$ of open neighborhoods of $0$ such that, for each $n \in \omega$, the following conditions are satisfied:
\\\indent\textup{(i)} $U_0 \subseteq U$;
\\\indent\textup{(ii)} $U_n = \ominus U_n$;
\\\indent\textup{(iii)}	$U_{n+1} \oplus U_{n+1} \subseteq U_n$; and
\\\indent\textup{(iv)} for each $x \in G$ and each $n \in \omega$, there is $k \in \omega$ such that $x \oplus (U_k \oplus (\ominus x)) \subseteq U_n$.
\end{lemma}

\begin{proof}
Put $\gamma = \{U\}$, and take a countable family $\gamma^*$ of open neighborhoods of $0$ satisfying conditions (1)-(4) of Lemma~\ref{t-4}. Then $U \in \gamma^*$.
We are going to define, by induction, a sequence of elements of $\gamma^*$.

Let us first enumerate the elements of $\gamma^*$, say, $\gamma^* = \{W_n : n \in \omega\}$.
Since $U \in \gamma^*$, $W_0 \in \gamma^*$, and $\gamma^*$ satisfies conditions (2) and (3) of Lemma~\ref{t-4}, it follows that $U \cap W_0 \in \gamma^*$ and there exists $U_0 \in \gamma^*$ such that $U_0 \subseteq U \cap W_0$.
Now assume that, for some $n \in \omega$, an element $U_n \in \gamma^*$ has already been defined.

Then since $\gamma^*$ satisfies conditions (2) and (3), we can choose an element $V \in \gamma^*$ such that $V = \ominus V$ and $V \oplus V \subseteq U_n \cap \bigcap_{i=0}^n W_i$. Put $U_{n+1} = V$. The definition is complete. Thus, the sequence $\{U_n : n \in \omega\}$ satisfies conditions (i)-(iii).

Fix $n \in \omega$ and $x \in G$. Since the family $\gamma^*$ satisfies condition (4) of Lemma~\ref{t-4}, there exists $j \in \omega$ such that $x \oplus (W_j \oplus (\ominus x)) \subseteq U_n$. Put $k = \max\{n, j\}$. From the definition of $U_{k+1}$, it follows that $U_{k+1} \subseteq U_{k+1} \oplus U_{k+1} \subseteq W_j$. Thus,
$$ x \oplus (U_{k+1} \oplus (\ominus x)) \subseteq x \oplus (W_j \oplus (\ominus x)) \subseteq U_n. $$
Thus the proof is completed.
\end{proof}

The following two lemmas were proved in \cite{3} and~\cite{4} respectively.

\begin{lemma}\cite{3}\label{t-6}
Let $(G,\mathscr{T},\oplus)$ be a topological gyrogroup, and let $H \unlhd G$. Then $(G/H,\mathfrak{T},\oplus)$ is a topological quotient gyrogroup, where $\mathfrak{T}$ is the quotient topology.
\end{lemma}

\begin{lemma}\cite{4}\label{t-7}
Let $(G, \tau, \oplus)$ be a topological gyrogroup. If $U$ is an open neighborhood of $0$ and $F$ is a compact subset of $G$, then there exists an open neighborhood $V$ of $0$ such that $(a \oplus V) \oplus (b \oplus V) \subseteq (a \oplus b) \oplus U$ for every $a, b \in F$.
\end{lemma}

\begin{lemma}\label{t-8}
Let $(G, \tau, \oplus)$ be a strongly topological gyrogroup with a symmetric neighborhood base $\mu$ at $0$. If $U \in \mu$ and $H$ is a compact subset of $G$, then there exists $V \in \mu$ such that $(h \oplus V) \boxplus (\ominus h) \subseteq U$, $(\ominus h) \oplus (V \oplus h) \subseteq U$  and $h\oplus (V \boxplus (\ominus h))\subseteq U$ for every $h \in H$.
\end{lemma}

\begin{proof}
Take any open neighborhood $U$ of $0$.  By \cite{4}, there exists $V \in \mu$ such that $(h \oplus V) \boxplus (\ominus h) \subseteq U$ and $(\ominus h) \oplus (V \oplus h) \subseteq U$ for every $h \in H$. Hence it suffices to prove that there exists $V \in \mu$ such that $h\oplus (V \boxplus (\ominus h))\subseteq U$ for every $h \in H$. Indeed, for each $h\in H$, it follows from the definition of the operation ``$\boxplus$'' that there exists $V_{h}\in\mu$ such that $(h\oplus V_{h})\oplus (V_{h}\boxplus(\ominus(h\oplus V_{h})))\subseteq U$. Since $H$ is compact and $H\subseteq\bigcup_{h\in H}(h\oplus V_{h})$, there is a finite subset $\{h_{1},...,h_{n}\}\subset H$ such that $H\subseteq\bigcup_{k=1}^{n}(h_{k}\oplus V_{h_{k}})$. Put
\[V=\bigcap_{k=1}^{n}V_{h_{k}}.\]
    Clearly, $V$ is an open neighborhood of $0$ in $G$. For any $h\in H$, there exists $1\leq k\leq n$ such that $h\in h_{k}\oplus V_{h_{k}}$, hence
    \begin{align*}
h\oplus ( V \boxplus(\ominus h)) & \subseteq(h_{k}\oplus V_{h_{k}})\oplus (V\boxplus(\ominus(h_{k}\oplus V_{h_{k}}))) \\
        & \subseteq(h_{k}\oplus V_{h_{k}})\oplus (V_{h_{k}}\boxplus(\ominus(h_{k}\oplus V_{h_{k}}))) \\
        & \subseteq U.
    \end{align*}
\end{proof}

Finally, we give an important lemma, which provides a criterion to determine whether a subgyrogroup is an $L$-subgyrogroup.

\begin{lemma}\cite{5}\label{t-10}
Let $H$ be a subgyrogroup of a gyrogroup $G$. Then $H \trianglelefteq G$ if and only if the operation on the coset space $G/H$ given by $$(a \oplus H) \oplus (b \oplus H) = (a \oplus b) \oplus H$$ is well defined.
\end{lemma}

\section{The range-metrizability of $\sigma$-compact strongly topological gyrogroups}
In this section, we first prove that each $\sigma$-compact strongly  topological gyrogroup is range-metrizable. Then, by applying this result, we show that each $\sigma$-compact strongly topological gyrogroup $G$ is topologically isomorphic to a subgyrogroup of the topological product of some family of second-countable strongly topological gyrogroups.

\begin{thm}\label{t-11}
Let $G$ be a $\sigma$-compact strongly topological gyrogroup. For each open neighborhood $U$ of the identity element $0$ in $G$, there exists a continuous homomorphism $\pi$ of $G$ onto a  metrizable strongly topological gyrogroup $H$ such that $\pi^{-1}(\widetilde{V}) \subseteq U$, for some open neighborhood $\widetilde{V}$ of the identity element $e_H$ of $H$.
\end{thm}

\begin{proof}
Suppose that $G$ is a strongly topological gyrogroup with a symmetric neighborhood base $\mu$ at $0$ satisfying $\mathrm{gyr}[x,y](W) = W$ for each $W \in \mu$ and each $x, y \in G$. Let $U$ be an arbitrary open neighborhood of $0$ in $G$. Then there exists $V_0 \in \mu$ such that $V_0 \oplus V_0 \subseteq U$.

Since $G$ is $\sigma$-compact, we can assume $G = \bigcup_{n \in \omega}F_n$, where each $F_n$ is a compact subset of $G$, $0 \in F_0$ and $F_n \subseteq F_{n+1}$ for each $n \in \omega$.

By Lemmas~\ref{t-7} and~\ref{t-8}, there exists a subfamily $\{V_n: n \in \omega \} \subseteq \mu$ such that, for each $n \in \omega$, the following conditions hold:

\indent\textup{(i)} $V_{n+1} \oplus V_{n+1} \subseteq V_n \cap U$;

\indent\textup{(ii)} $\mathrm{gyr}[x,y](V_n) = V_n$ for each $x,y \in G$;

\indent\textup{(iii)} $(\ominus x) \oplus (V_{n+1} \oplus x) \subseteq V_n$ for each $x \in F_n$;

\indent\textup{(iv)} $(x \oplus V_{n+1}) \oplus (y \oplus V_{n+1}) \subseteq (x \oplus y) \oplus V_n$ for each $x,y \in F_n$;

\indent\textup{(v)}  $x \oplus ( V_{n+1} \boxplus (\ominus x)) \subseteq V_n$ for each $x \in F_n$.

Clearly, $\overline{V_{n+1}} \subseteq V_{n+1} \oplus V_{n+1} \subseteq V_n \cap U$ for each $n \in \omega$. Thus, $$Z = \bigcap_{n \in \omega}V_n = \bigcap_{n \in \omega}\overline{V_{n+1}}$$ is closed in $G$. It is obvious that $Z$ is a subgyrogroup of $G$.

Next we prove that $Z$ is a normal subgyrogroup of $G$. By Lemma~\ref{t-10}, it suffices to verify that $(x \oplus Z) \oplus (y \oplus Z) = (x \oplus y) \oplus Z$ for each $x, y \in G$. Now fix any $x, y \in G$. Then $$(x \oplus y) \oplus Z \subseteq x \oplus (y \oplus \mathrm{gyr}[y,x](Z)) = x \oplus (y \oplus Z)$$ by (ii), hence $(x \oplus y) \oplus Z \subseteq (x \oplus Z) \oplus (y \oplus Z)$. Moreover, by(iv), we have:
\begin{align*}
(x \oplus Z) \oplus (y \oplus Z) &= \left( \bigcap_{n\in\omega}(x \oplus V_n) \right) \oplus \left( \bigcap_{n\in\omega}(y \oplus V_n) \right) \\
&\subseteq \bigcap_{n\in\omega} ((x \oplus V_{n+1}) \oplus (y \oplus V_{n+1})) \\
&\subseteq \bigcap_{n\in\omega} ((x \oplus y) \oplus V_n) \\
&= (x \oplus y) \oplus Z.
\end{align*}
By the arbitrary choices of $x, y \in G$, the subgyrogroup $Z$ is normal.

Moreover, $$\mathrm{gyr}[x,y](Z) = \mathrm{gyr}[x,y](\bigcap_{n \in \omega} V_n) = \bigcap_{n \in \omega}\mathrm{gyr}[x,y](V_n) = \bigcap_{n \in \omega}V_n = Z,$$
hence $Z$ is a strong subgyrogroup of $G$.
It follows from Lemma~\ref{t-6} that $G/Z$ is a topological quotient gyrogroup.

Let $\pi: G \rightarrow G/Z$ be the natural mapping of $G$ onto $G/Z$.
Since $Z$ is normal, it follows that $\pi$ is a gyrogroup homomorphism.
We claim that $G/Z$ is a strongly topological gyrogroup. Clearly, $\pi(W)$ is symmetric for each $W\in\mu$.
Indeed, since $\pi$ is a gyrogroup homomorphism, we conclude that $\pi(W) = \pi(\ominus W) = \ominus (\pi(W))$.
By \cite[Theorem 3.13]{2}, it is easy to see that  the family $\{\pi(W): W \in \mu \}$ is a symmetric neighborhood base of the gyrogroup $G/Z$ at the identity element of $G/Z$.
Now take arbitrary $\pi(a)$, $\pi(b) \in G/Z$ and $W \in \mu$, where $a, b\in G$; then it follows from \cite[Proposition 23]{6} that $$\mathrm{gyr}[\pi(a),\pi(b)](\pi(W)) = \pi(\mathrm{gyr}[a,b](W)) = \pi(W).$$
Thus, $G/Z$ is a strongly topological gyrogroup.

By Lemma~\ref{t-3}, there exists a continuous prenorm $N$ on $G$ which satisfies $$N(\mathrm{gyr}[x,y]z) = N(z)$$ for each $x, y, z \in G$ and
$$\{x \in G: N(x) < 1/2^n \} \subseteq V_n \subseteq \{x \in G: N(x) \leq 2/2^n \}$$
for each integer $n \geq 0$. Next we prove the following 6 claims.

\smallskip
\textbf{Claim 1:} $Z = \{x \in G: N(x) = 0 \}$.

Let $x$ be an arbitrary element of $Z$. Since $Z = \bigcap_{n \in \omega}V_n$ and $V_n \subseteq \{x \in G: N(x) \leq 2/2^n \}$ for each $n \in \omega$, it follows that $N(x) = 0$. Conversely, suppose that $x \in G$ and $N(x) = 0$. Then for each $n \in \omega$, $N(x) < 1/2^n$, hence $x \in \bigcap\{V_n: n \in \omega \} = Z$.

\smallskip
\textbf{Claim 2:} $N(x \oplus z) = N(x) = N(z \oplus x)$ for each $x \in G$ and $z \in Z$.

Indeed, for each $x \in G$ and $z \in Z$, $N(x \oplus z) \leq N(x) \oplus N(z) = N(x) \oplus 0 = N(x)$.
By the definition of $N$, for each $x, y ,z \in G$, we have $N(\mathrm{gyr}[x,y](z)) = N(z)$. Then
\begin{align*}
N(x) &= N(x \oplus (z \oplus (\ominus z)))\\
&= N((x \oplus z) \oplus \mathrm{gyr}[x,z](\ominus z))\\
&\leq N(x \oplus z) \oplus N(\mathrm{gyr}[x,z](\ominus z))\\
&= N(x \oplus z) \oplus N(\ominus z)\\
&= N(x \oplus z)
\end{align*}
Thus, $N(x \oplus z) = N(x)$ for each $x \in G$ and $z \in Z$.

Moreover, $N(z \oplus x) \leq N(z) + N(x) = 0 + N(x) = N(x)$
and
$$N(x) = N(\ominus z \oplus (z \oplus x)) \leq N(\ominus z) + N(z \oplus x) = N(z \oplus x),$$
hence $N(z \oplus x) = N(x)$ for each $x \in G$ and $z \in Z$.

Now define a function $\varrho: G \times G \rightarrow \mathbb{R}$ by
$$\varrho(x,y) = N(\ominus x \oplus y)$$
for any $x,y \in G$.
It is clear that $\varrho$ is continuous.

\smallskip
\textbf{Claim 3.} The function $\varrho$ is a pseudometric on the set $G$.

(I) If $x = y$, then $\varrho(x, y) = N(\ominus x \oplus y) = 0$ for all $x \in G$.

(II)$\varrho(x, y) = N(\ominus x \oplus y) =N(\ominus(\ominus x \oplus y)) = N(\mathrm{gyr}[\ominus x, y](\ominus y \oplus x))=N(\ominus y \oplus x)= \varrho(y, x)$  for all $x, y \in G$.

(III) For any $x, y, z \in G$, we have
\begin{align*}
\varrho(x,z) &=N(\ominus x \oplus z)\\
&=N( \ominus x \oplus (y \oplus (\ominus y \oplus z)))\\
&= N((\ominus x \oplus y) \oplus \mathrm{gyr}[\ominus x, y](\ominus y \oplus z)) \\
&\le N(\ominus x \oplus y) + N(\mathrm{gyr}[\ominus x, y](\ominus y \oplus z))\\
&\le N(\ominus x \oplus y) + N(\ominus y \oplus z)\\
&=\varrho(x, y) + \varrho(y, z).
\end{align*}
Thus $\varrho$ is a pseudometric on $G$.

\smallskip
\textbf{Claim 4.} The pseudometric $\varrho$ is left invariant.

Indeed, for any $x, y, z\in G$, it is obvious that $\varrho(z \oplus x, z \oplus y) = N\big(\ominus(z \oplus x) \oplus (z \oplus y)\big)$. Since $$\ominus(z \oplus x) \oplus (z \oplus y) = \text{gyr}[z, x](\ominus x \oplus y),$$
it follows that $$\varrho(z \oplus x, z \oplus y) = N\Big(\text{gyr}[z, x](\ominus x \oplus y)\Big) = N(\ominus x \oplus y) = \varrho(x, y).$$
Therefore, $\varrho$ is left-invariant.

\smallskip
\textbf{Claim 5.} For each $a, b \in G$, $a_1 \in a \oplus Z$ and $b_1 \in b \oplus Z$, the equation $\varrho(a_1, b_1) = \varrho(a, b)$ holds.

Since $a_1 \in a \oplus Z$, there exists $z_1 \in Z$ such that $a_1 = a \oplus z_1$, then
$$\varrho(a, a_1) = \varrho(a, a \oplus z_1) = N(\ominus a \oplus (a \oplus z_1)) = N(z_1) = 0.$$ Therefore, we have
$$\varrho(a_1, b) \leq \varrho(a_1, a) + \varrho(a, b) = 0 + \varrho(a, b) = \varrho(a, b),$$ and then
$$\varrho(a, b) \leq \varrho(a, a_1) + \varrho(a_1, b) = 0 + \varrho(a_1, b) = \varrho(a_1, b).$$
Thus, $\varrho(a_1, b) = \varrho(a, b)$.

Similarly, we can prove that $\varrho(a_1, b) = \varrho(a_1, b_1)$. Therefore, $\varrho(a_1, b_1) = \varrho(a, b)$.

Next we define a function $d$ on $G/Z \times G/Z$ by $$d(\pi(x), \pi(y)) = \varrho(x, y)$$
for each $x, y \in G$. It is clear that $d$ is continuous.
By Claim 5, the function $d$ is well defined.

\smallskip
\textbf{Claim 6.} The function $d$ is a metric on $G/Z$.

(a) For any $x, y\in G$, it is obvious that $d(\pi(x), \pi(y)) = \varrho(x, y) = N(\ominus x \oplus y) \ge 0$.
If $d(\pi(x), \pi(y)) = 0$, then $\varrho(x, y) = N(\ominus x \oplus y) = 0$, which means $\ominus x \oplus y \in Z$.
Thus $y \in x \oplus Z$, then $\pi(x)=\pi(y)$. Conversely, if $\pi(x) = \pi(y)$, then $d(\pi(x), \pi(y)) = d(\pi(x), \pi(x)) = N(\ominus x \oplus x) = N(0) = 0$.

(b) For any $x, y \in G$, we have
\begin{align*}
d(\pi(x), \pi(y)) &= \varrho(x, y) \\
&= N(\ominus x \oplus y) \\
&= N(\ominus(\ominus x \oplus y)) \quad \text{(symmetry of the prenorm)} \\
&= N(\mathrm{gyr}[\ominus x, y](\ominus y \oplus x)) \\
&= N(\ominus y \oplus x) \quad \text{(since } N(\mathrm{gyr}[a, b]v) = N(v)\text{)} \\
&= \varrho(y, x) \\
&= d(\pi(y), \pi(x)).
\end{align*}

(c) For any $x, y, z \in G$, we have
\begin{align*}
d(\pi(x), \pi(z)) &= \varrho(x, z) \\
&\le \varrho(x, y) + \varrho(y, z) \\
&= d(\pi(x), \pi(y)) + d(\pi(y), \pi(z)).
\end{align*}

Therefore, the function $d$ is a metric on the quotient space $G/Z$. Let $H=G/Z$, and let $E$ be the identity element of $H$.
We also define a function $N_H$ on $H$ by the rule $$N_H(A) = N(a)$$ for each $A \in H$ and $a \in A$. Then
$N_H(\pi(x)) = N(x)$ for each $x \in G$. From Claim 2, the function $N_H$ is well defined.

Since $\pi$ is a gyrogroup homomorphism of $G$ onto $H$ and $\pi(Z)$ is the identity element of the gyrogroup $H$, one easily verifies that $N_H$ is a prenorm on $H$ satisfying the additional condition: If $N_H(A) = 0$, then $A$ is the identity element of $H$.

For $\varepsilon > 0$, put $B(\varepsilon) = \{x \in G: N(x) < \varepsilon \}$ and $O(\varepsilon) = \{X \in H: N_H(X) < \varepsilon \}$;
then it is obvious that $\pi(B(\varepsilon)) = O(\varepsilon)$.

We claim that the prenorm $N$ also satisfies that for each $\varepsilon > 0$ and each $x \in G$, there exists $\delta > 0$ such that $x \oplus (B(\delta) \oplus (\ominus x)) \subseteq B(\varepsilon).$

Indeed, for each $\varepsilon > 0$, there exists $n \in \omega$ such that $2/2^n < \varepsilon$.
If $N(x) < 2/2^n$, then $N(x) < \varepsilon$.
So $V_n \subseteq \{x \in G: N(x) \leq 2/2^n \} \subseteq B(\varepsilon)$.
By (ii), there exists $k \in \omega$ such that $x \oplus (V_k \oplus (\ominus x)) \subseteq V_n$.
Moreover, we have $\{x \in G: N(x) < 1/2^k \} \subseteq V_k$. Put $\delta = 1/2^k$. Then $B(\delta) \subseteq V_k$.
Hence $x \oplus (B(\delta) \oplus (\ominus x)) \subseteq x \oplus (V_k \oplus (\ominus x)) \subseteq V_n \subseteq B(\varepsilon)$.

Moreover, it is easy to verify that the following (d1)-(d3) hold.

\smallskip
(d1) For each $\varepsilon > 0$ and any $X \in H$, there exists $\delta > 0$ such that $X \oplus (O(\delta) \oplus (\ominus X)) \subseteq O(\varepsilon).$

\smallskip
(d2) For each $X \in H$, we have $d(E, X) = d(E, \ominus X)$.

\smallskip
(d3) For each $n\in\mathbb{N}$, we have $O(1/2^{n+1}) \oplus O(1/2^{n+1}) \subseteq O(1/2^n)$.

Now let $\mathscr{T}_H$ be the topology generated by the metric $d$ on $H$. Let us show that $H$ with this topology is a topological gyrogroup. Indeed, it suffices to prove that the family $\mathscr{U}=\{O(1/2^n): n \in \omega \}$ is a base of the space $H$ at the identity element $E$, which satisfies the conditions (1)-(9) of \cite[Theorem 4.4]{9}.

(1) For every $O(1/2^n) \in \mathscr{U}$, it follows from (d3) that $O(1/2^{n+1}) \oplus O(1/2^{n+1}) \subseteq O(1/2^n)$.

(2) For any $O(1/2^n) \in \mathscr{U}$ and $X \in O(1/2^n)$, there exists an element $V = O(1/2^k) \in \mathscr{U}$ for some $k \geq n$  such that $X \oplus V \subseteq O(1/2^n)$.

Indeed, for each $X \in O(1/2^n)$, we have $N_H(X) = N(x) <1/2^n$ for any $x\in X$. Fix any $x\in X$ and put $\delta = 1/2^n - N(x)$. Then there exists $k \geq n$ such that $1/2^k < \delta$. For each $y \in B(1/2^k)$, we have $N(y) < 1/2^k$. Then $N(x \oplus y) \leq N(x) + N(y) < N(x) + 1/2^k < N(x) + \delta = N(x) + 1/2^n - N(x) = 1/2^n$, hence $x \oplus y \in B(1/2^n)$, which shows $x \oplus B(1/2^k) \subseteq B(1/2^n)$. Since $\pi$ is a gyrogroup homomorphism, we conclude that
\[
X \oplus O(1/2^k)=\pi(x) \oplus \pi(B(1/2^k))=\pi(x \oplus B(1/2^k))\subseteq \pi(B(1/2^n))=O(1/2^n).
\]

(3)  For any $O(1/2^n) \in \mathscr{U}$ and $X \in H$, it follows from (d1) that for each $\varepsilon > 0$ and $X \in H$, there exists $\delta> 0$ such that $X \oplus (O(\delta) \oplus (\ominus X)) \subseteq O(\varepsilon).$

(4) For any $n, m\in \omega$, we have $O(1/2^k) \subseteq O(1/2^n) \cap O(1/2^m)$, where $k=\max \{m, n\}$.

Indeed, for each $X \in O(1/2^k)$, we have $N_H(X) < 1/2^k$, then $N_H(X) < 1/2^m$ and $N_H(X) < 1/2^n$, which shows that $X \in O(1/2^n)\cap O(1/2^m)$. Thus $O(1/2^k)\subseteq O(1/2^n) \cap O(1/2^m)$.

(5) For any $O(1/2^n) \in \mathscr{U}$ and $A, C \in H$, we have $\mathrm{gyr}[A, C]O(1/2^n)= O(1/2^n)$.

Indeed, for each $x \in B(1/2^n)$, we have $N(x) < 1/2^n$, then $N(\mathrm{gyr}[a, c]x)=N(x)< 1/2^n$, where $a, c\in G$ and $\pi(a)=A$ and $\pi(c)=C$. Hence $\mathrm{gyr}[a, c]x \in B(1/2^n)$. By the arbitrary choice of $x$, we conclude that $\mathrm{gyr}[a, c]B(1/2^n) \subseteq B(1/2^n)$. Since $\mathrm{gyr}[a, c]$ is a bijective, it follows that $\mathrm{gyr}[a, c]B(1/2^n)=B(1/2^n)$. Now we have
\begin{align*}
\mathrm{gyr}[A, C]O(1/2^n)
&= \mathrm{gyr}[\pi(a), \pi(c)]\pi(B(1/2^n)) \\
&= \pi(\mathrm{gyr}[a, c]B(1/2^n))\\
&= \pi(B(1/2^n))=O(1/2^n).
\end{align*}

(6) For every $O(1/2^n) \in \mathscr{U}$ and $B \in H$, we have $$\bigcup_{W \in O(1/2^n)}\mathrm{gyr}[W, B]O(1/2^n) \subseteq O(1/2^n).$$

Indeed, pick $w\in G$ such that $\pi(w)=W$ for each $W\in O(1/2^n)$, and pick $b\in G$ such that $\pi(b)=B$. Then
\begin{align*}
\bigcup_{W \in O(1/2^n)}\mathrm{gyr}[W, B]O(1/2^n)
&= \bigcup_{W \in V}\mathrm{gyr}[\pi(w), \pi(b)]\pi(B(1/2^n)) \\
&= \bigcup_{W \in V}\pi(\mathrm{gyr}[w, b]B(1/2^n))\\
&= \pi(B(1/2^n))\\
&= O(1/2^n).
\end{align*}

(7) The equation $\{E\} = \bigcap_{n \in \omega}(O(1/2^n) \boxminus O(1/2^n))$ holds.

Indeed, it is obvious that $E = E \boxminus E \in \bigcap_{n \in \omega}(O(1/2^n) \boxminus O(1/2^n))$. Now take any $X \in \bigcap_{n \in \omega}(O(1/2^n) \boxminus O(1/2^n))$, it suffices to show that $X=E.$ Clearly, $X = Y_n \boxminus Z_n$, where $Y_n \in O(1/2^n)$, $Z_n \in O(1/2^n)$ for each $n \in \omega$. Then $N_H(Y_n) < 1/2^n$ and $N_H(Z_n) < 1/2^n$ for each $n \in \omega$. Therefore, for each $n \in\omega$, we have
\begin{align*}
N_H(X) &= N_H(Y_n \boxminus Z_n)\\
&= N_H(Y_n \oplus \mathrm{gyr}[Y_n, Z_n](\ominus Z_n))\\
&\leq N_H(Y_n) + N_H(\mathrm{gyr}[Y_n, Z_n](\ominus Z_n)))\\
&= N_H(Y_n) + N_H(\mathrm{gyr}[\pi(y_n), \pi(z_n)](\ominus \pi(z_n))) \quad \text{(where $\pi(y_n) = Y_n$, $\pi(z_n) = Z_n$)}\\
&= N_H(Y_n) + N_H(\pi(\mathrm{gyr}[y_n, z_n](\ominus z_n)))\\
&= N(y_n) + N(\mathrm{gyr}[y_n, z_n](\ominus z_n))\\
&= N(y_n) + N(\ominus z_n)\\
&= N(y_n) + N(z_n)\\
&\leq 1/2^n + 1/2^n\\
&= 1/2^{n-1}.
\end{align*}
Then $N_H(X) = 0$. By the definition of $N_{H}$, we have $X= E$. Therefore, $\{E\} = \bigcap_{n \in \omega}{(O(1/2^n) \boxminus O(1/2^n))}$.

(8) For every $O(1/2^n) \in \mathscr{U}$ and $X \in H$, there exists $V \in \mathscr{U}$ such that $V \boxplus X \subseteq X \oplus O(1/2^n)$ and $X \oplus V \subseteq X \boxplus O(1/2^n)$.

Take any $x\in G$ such that $\pi(x)=X$.
Clearly, we have $$V_{n+2} \subseteq \{g \in G: N(g) \leq 2/2^{n+2} \} \subseteq \{g \in G: N(g) \leq 1/2^{n} \}.$$
By (v), there exists $k \in \omega$ such that $(\ominus x) \oplus (V_k \boxplus x) \subseteq V_{n+2}$, then we have $\{g \in G: N(g) < 1/2^k \} \subseteq V_k$. Put $\delta = 1/2^k$; then $B(\delta) \subseteq V_k$.
Hence $$(\ominus x) \oplus (B(\delta) \boxplus x) \subseteq (\ominus x) \oplus (V_k \boxplus x) \subseteq V_{n+2}.$$ Therefore,
\begin{align*}
\pi((\ominus x) \oplus (B(\delta) \boxplus x))
&= \ominus \pi(x) \oplus (\pi(B(\delta)) \boxplus \pi(x))\\
&=(\ominus X) \oplus (O(\delta ) \boxplus X) \\
&\subseteq \pi(B(1/2^n))\\
&=O(1/2^n).
\end{align*}
Thus, $O(\delta) \boxplus X \subseteq X \oplus O(1/2^n)$. Now put $V_{1} = O(1/2^k)$. Then we have $V \boxplus X \subseteq X \oplus O(1/2^n)$.

Similarly, we can find $V_{2}=O(1/2^l)$ for some $\in\mathbb{N}$ such that $X \oplus V_{2} \subseteq X \boxplus O(1/2^n)$. Now put $m=\max\{k, l\}$ and $V=O(1/2^m)$. Then we have $V \boxplus X \subseteq X \oplus O(1/2^n)$ and $X \oplus V \subseteq X \boxplus O(1/2^n)$.

(9) For every $U=O(1/2^n) \in \mathscr{U}$, since $O(1/2^n)=\ominus O(1/2^n)$, it follows that  $\ominus U\in \mathscr{U}$.

Therefore we have proved that the family $\{O(1/2^n): n \in \omega \}$, which is a base of the space $H$ at the identity element $E$, satisfies conditions (1)-(9) of \cite[Theorem 4.4]{9}. Thus, $H$ with the topology $\mathscr{T}_H$ is a strongly topological gyrogroup. Since the family $\{O(1/2^n): n \in \omega \}$ is countable, it follows that $H$ with the topology $\mathscr{T}_H$ is metrizable.

Finally, the equation $\pi(B(\varepsilon)) = O(\varepsilon)$ for any $\varepsilon > 0$, which implies that the homomorphism $\pi$ of $G$ onto $H = G/Z$ is continuous at the identity element of $G$. Since $G$ and $H$ are topological gyrogroups, it follows that $\pi$ is continuous. Notice also that if $x \in G$, $X = \pi(x)$, and $\varepsilon > 0$, then $N(x) < \varepsilon$ is equivalent to $N_H(X) < \varepsilon$. Therefore, $\pi^{-1}(O(\varepsilon)) = B(\varepsilon)$ for each $\varepsilon > 0$. In particular, $\pi^{-1}(O(1)) = B(1) \subseteq U$. Now put $\widetilde{V}=O(1)$. The proof is completed.
\end{proof}

\begin{defn}\label{t-15}
let $\mathscr{P}$ be a class of topological gyrogroups, and let $G$ be any topological gyrogroup. Let us say that $G$ is {\color{blue} range-$\mathscr{P}$} if for every open neighborhood $U$ of the identity element $0_G$ of $G$, there exists a continuous homomorphism $\pi$ of $G$ to a gyrogroup $H \in \mathscr{P}$ such that $\pi^{-1}(V) \subset U$, for some open neighborhood $V$ of the identity element $0_H$ of $H$.
\end{defn}

Now, by Theorem~\ref{t-11}, we prove one of main theorems in this section as follows.

\begin{thm}
Each $\sigma$-compact strongly topological gyrogroup $G$ is range-metrizable.
\end{thm}

The following two propositions are proved in \cite{13}.

\begin{prop}\cite{13}\label{t-12}
If a topological gyrogroup $H$ is a continuous homomorphic image of an $\omega$-narrow topological gyrogroup $G$, then $H$ is also $\omega$-narrow.
\end{prop}

\begin{prop}\cite{13}\label{t-13}
Let $(G, \tau, \oplus)$ be a left $\omega$-narrow strongly topological gyrogroup with a symmetric open neighborhood base $\mathscr{U}$ at $0$. If $G$ is first-countable, then $G$ has a countable base.
\end{prop}

\begin{cor}\label{t-14}
Suppose that $G$ is a $\sigma$-compact strongly topological gyrogroup. For every open neighborhood $U$ of $0$, there exists a continuous homomorphism $\pi$ from $G$ onto a second-countable strongly topological gyrogroup $H$ such that $\pi^{-1}(V) \subseteq U$ for some open neighborhood $V$ of the identity element in $H$.
\end{cor}

\begin{proof}
By Theorem~\ref{t-11}, one can find a continuous homomorphism $\pi$ of $G$ onto a metrizable strongly topological gyrogroup $H$ and an open neighborhood $V$ of the identity in $H$ such that $\pi^{-1}(V) \subset U$. From Proposition~\ref{t-12}, it follows that the gyrogroup $H$ is $\omega$-narrow, so Proposition~\ref{t-13} implies that $H$ is second-countable.
\end{proof}

It is well known that each $\omega$-narrow topological group is range-second-countable, see \cite[Corollary 3.4.19]{AT}. Therefore, we have the following question.

\begin{ques}\label{qqqq}
Let $G$ be an $\omega$-narrow strongly topological gyrogroup, is $G$ range-second-countable?
\end{ques}

If we can prove that each right $w$-balanced strongly topological gyrogroup $G$ is range-metrizable, then, by Lemma~\ref{t-2}, we can conclude that each right $w$-narrow strongly topological gyrogroup is range-second-countable, which gives a generalization of Corollary~\ref{t-14}. Hence, we have the following question.

\begin{ques}
Let $G$ be an $w$-balanced strongly topological gyrogroup, is $G$ range-second-countable?
\end{ques}

Clearly, it is easy to see that every subgyrogroup of a range-$\mathscr{P}$ gyrogroup is also range-$\mathscr{P}$.

\begin{prop}\label{t-16}
Let $\mathscr{P}$ be any class of topological gyrogroups which is closed under finite products, and let $H$ be the topological product of a family $\{H_a: a \in A \}$ of gyrogroups in the class $\mathscr{P}$. Then every subgyrogroup of $H$ is range-$\mathscr{P}$.
\end{prop}

\begin{proof}
Let $0$ be the identity element of the product gyrogroup $H = \prod_{\alpha \in A} H_\alpha$, and let $U$ be an arbitrary open neighborhood of $0$ in $H$.

By the definition of product topology, there exist a finite subset $F \subseteq A$ and an open neighborhood $V_\alpha$ of the identity element $0_\alpha$ in $H_\alpha$ for each $\alpha \in F$ such that the canonical open set $\left( \prod_{\alpha \in F} V_a \right) \times \left( \prod_{a \notin F} H_a \right) \subseteq U$.

Let $\pi_F : H \to \prod_{\alpha \in F} H_\alpha$ be the projection mapping.
Since each $H_\alpha \in \mathscr{P}$ and the class $\mathscr{P}$ is closed under finite products,
it follows that $\prod_{\alpha \in F} H_\alpha \in \mathscr{P}$.
Moreover, it is obvious that $\pi_F$ is a continuous homomorphism.
Let $V = \prod_{\alpha \in F} V_\alpha$ be an open neighborhood of the identity element in $\prod_{\alpha \in F} H_\alpha$. Then $\pi_F^{-1}(V) \subseteq U$.
Therefore, by definition, the gyrogroup $H$ is range-$\mathscr{P}$.
Since $H$ is range-$\mathscr{P}$, it follows that every subgyrogroup of $H$ is range-$\mathscr{P}$.
\end{proof}

\begin{thm}\label{t-17}
Let $\mathscr{P}$ be a class of (strongly) topological gyrogroups, $\tau$ an infinite cardinal number, and $G$ a (strongly) topological gyrogroup, which is range-$\mathscr{P}$ and has a base $\mathscr{B}$ consisting of open neighborhoods of the identity element of $G$ such that $|\mathscr{B}| \leq \tau$. Then $G$ is topologically isomorphic to a subgyrogroup of the product of a family $\{H_a: a \in A \}$ of groups such that $H_a \in \mathscr{P}$, for each $a \in A$, and $|A| \leq \tau$.
\end{thm}

\begin{proof}
Fix a base $\mathscr{B}$ consisting of open neighborhoods of the identity element in $G$ such that $|\mathscr{B}| \leq \tau$.
By Theorem~\ref{t-11}, we can choose, for every $U \in \mathscr{B}$, a continuous homomorphism $f_U$ of $G$ onto a (strongly) topological gyrogroup $H_U \in \mathscr{P}$ such that $(f_U)^{-1}(V_{U}) \subseteq U$, for some open neighborhoods $V_{U}$ of identity element in $H$. We claim that the family $\{f_U: U \in \mathscr{B}\}$ separates points and closed sets.

By the homogeneity of $G$, we only need to consider $x=e$ and $e \notin F$, where $F$ is a closed subset of $G$.
Then $G \setminus F$ is an open neighborhood of $e$.
Since $\mathscr{B}$ is a local base of $e$ in $G$, there exists some $U \in \mathscr{B}$ satisfying $e \in U \subseteq G \setminus F$.
Then $f_U^{-1}(V_{U}) \subseteq U \subseteq G \setminus F$, hence $f_U^{-1}(V_{U}) \cap F = \emptyset$, which means $f_U(f_U^{-1}(V_{U})\cap F)= V_{U} \cap f_U(F) = \emptyset$. Therefore, $f_U(e)$ does not belong to the closure of $f_U(F)$.

Now, it is easy to see that the diagonal product $h$ of the family $\{f_U: U \in \mathscr{B}\}$ is a topological isomorphism of $G$ onto a (strongly) topological subgyrogroup of the topological product of the family $\{H_U: U \in \mathscr{B}\}$.
\end{proof}

By Theorems~\ref{t-16} and~\ref{t-17}, we have the second main theorem in this section.

\begin{thm}\label{t-18}
Let $G$ be a strongly topological gyrogroup. Then $G$ is range-metrizable if and only if $G$ is topologically isomorphic to a subgyrogroup of a topological product of metrizable strongly topological gyrogroups.
\end{thm}

By Theorems~\ref{t-11} and ~\ref{t-18}, we have the following theorem.

\begin{thm}\label{t-19}
If $G$ is a $\sigma$-compact strongly topological gyrogroup, then $G$ is topologically isomorphic to a subgyrogroup of the topological product of some family of second-countable strongly topological gyrogroups.
\end{thm}

Finally, we give a sufficient condition of $rinv(G) \leq \omega$ for a a strongly topological gyrogroup $G$, which is a complementary to Theorem~\ref{t-2} and Question~\ref{qqqq} respectively.

\begin{thm}\label{t-2000}
Let $G$ be a strongly topological gyrogroup. If $G$ is range-metrizable, then $rinv(G) \leq \omega$.
\end{thm}

\begin{proof}
Take any open neighborhood $U$ of the identity element $e_G$ in a range-metrizable gyrogroup $G$. Then there exists a continuous homomorphism $p: G \rightarrow H$ of $G$ onto a metrizable gyrogroup $H$ such that $p^{-1}(V) \subseteq U$ for some open neighborhood $V$ of $e_H$ in $H$.
Let $\mathscr{B}$ be a countable base at $e_H$ in $H$. We claim that the family $\{p^{-1}(U): U\in\mathscr{B}\}$ is right-subordinated to $V$.

Indeed, take any $x\in G$ and put $y=p(x)$.
Then there exists an open neighborhood $O \in \mathscr{B}$ such that $y \oplus (O \oplus (\ominus y)) \subseteq V$. Since $p$ is a homomorphism, it follows that $$p(x \oplus (p^{-1}(O) \oplus (\ominus x)))=p(x) \oplus (p(p^{-1}(O)) \oplus p(\ominus x))=y \oplus (O \oplus (\ominus y))\subseteq V.$$
Therefore, $x \oplus (p^{-1}(O) \oplus (\ominus x)) \subseteq p^{-1}(V) \subseteq U$. Then the family $\{p^{-1}(U): U\in\mathscr{B}\}$ is right-subordinated to $V$.

Moreover, the family $\{p^{-1}(O): O \in \mathscr{B}\}$ is a countable family, thus, $rinv(G) \leq \omega$. The proof is completed.
\end{proof}

\section{Isomorphic embedding of $\sigma$-compact strongly topological gyrogroups}
In this section, we first prove that a strongly topological gyrogroup $H$ is topologically isomorphic to a subgyrogroup of a separable strongly topological gyrogroup if $H$ is $\sigma$-compact and satisfies $\omega$(H) $\leq$ $c$. Then, by applying this result, we establish that some necessary and sufficient conditions for any $\sigma$-compact strongly topological gyrogroup $G$ which is a subgyrogroup of a separable strongly topological gyrogroup in Theorem~\ref{t-23}.

First, we give two technical lemmas and some definitions.

\begin{lemma}\cite{14}\label{t-21}
If $X$ is a separable regular space, then $w(X) \le \mathfrak{c}$. More generally, every regular space $X$ satisfies $w(X) \le 2^{d(X)}$.
\end{lemma}

\begin{lemma}\label{t-20}
If a strongly topological gyrogroup $H$ is $\sigma$-compact and satisfies $\omega$(H) $\leq$ $\mathfrak{c}$, then $H$ is topologically isomorphic to a subgyrogroup of a separable strongly topological gyrogroup.
\end{lemma}

\begin{proof}
Assume that an $\sigma$-compact gyrogroup $H$ satisfies $\omega$(H) $\leq$ $\mathfrak{c}$. It follows from Theorem~\ref{t-19} that $H$ is topologically isomorphic to a subgyrogroup of a topological product $\Pi = \prod_{i \in I} G_i$, where the cardinality of $I$ is at most $\mathfrak{c}$ and each $G_i$ is a second countable strongly topological gyrogroup. Then the gyrogroup $\Pi$ is separable by the Hewitt-Marczewski-Pondiczery theorem.
\end{proof}

\begin{defn}\cite{11}\label{t-22}
Let $G$ be a gyrogroup with identity $e$ and let $J = [0,1)$. A function $f: J \to G$ is a {\color{blue} step function} if there are real numbers $a_0, a_1, \dots, a_n$ such that $0 = a_0 < a_1 < \dots < a_n = 1$ and $f$ is constant on $[a_k, a_{k+1})$ for all $k = 0, 1, \dots, n - 1$, where we say $A = \{a_0, a_1, \dots, a_n\}$ is a partition of $J$.
\end{defn}

Denote by $G^\bullet$ the set of all step functions of a gyrogroup $G$. Now we can define an operation ``$\oplus^\bullet$'' on $G^\bullet$ by
$(f \oplus^\bullet g)(r) = f(r) \oplus g(r)$, for each $r \in J$ and all $f, g \in G^\bullet$. Then $(G^\bullet, \oplus^\bullet)$ forms a gyrogroup.

Now we prove main theorem in this section. The proof is similar to \cite[Theorem 3.2]{12}; however, we provide the proof for the convenience for the reader.

\begin{thm}\label{t-23}
The following statements are equivalent for an arbitrary $\sigma$-compact strongly topological gyrogroup $G$.
\begin{itemize}
\item[(a)] $G$ is homeomorphic to a subspace of a separable regular space;
\item[(b)] $G$ is topologically gyrogroup isomorphic to a subgyrogroup of a separable strongly topological gyrogroup;
\item[(c)] $G$ is topologically gyrogroup isomorphic to a closed subgyrogroup of a separable path-connected, locally path-connected strongly topological gyrogoup.
\end{itemize}
\end{thm}

\begin{proof}
Since each $T_1$ topological gyrogroup is regular, it is easy to see that ($c$) implies ($b$) and ($b$) implies ($a$). Hence it suffices to show that ($a$) implies ($c$).

Assume that a $\sigma$-compact strongly topological gyrogroup $G$ is homeomorphic to a subspace of a separable regular space $X$. Then $\omega(X) \leq \mathfrak{c}$ by Lemma~\ref{t-21}, hence $\omega(G) \leq \omega(X) \leq \mathfrak{c}$.
Applying Lemma~\ref{t-20}, we can find a separable strongly topological gyrogroup $H$ containing $G$ as a topological subgyrogroup. Clearly, $G$ can fail to be closed in $H$, so our next step is to construct another separable strongly topological gyrogroup containing $G$ as a closed subgyrogroup.
 Consider the set $H^\bullet$ of all functions $f$ on $J = [0,1)$ with values in $H$ such that, for some sequence $0 = a_0 < a_1 < \dots < a_n = 1$, the function $f$ is constant on $[a_k, a_{k+1})$ for each $k = 0, \dots, n-1$. Then $(H^\bullet, \oplus^\bullet)$ forms a gyrogroup after Definition~\ref{t-22}. The elements of $H^\bullet$ are called {\it step functions}.

The gyrogroup $H$ is gyrogroup isomorphic to a closed subgyrogroup of $H^\bullet$. The corresponding gyrogroup monomorphism $i: H \rightarrow H^\bullet$ assigns to each $h \in H$ the constant function $i(h)=h^\bullet \in H^\bullet$ defined by $h^\bullet(x)=h$ for all $x \in J$. Consequently, $i(H)$ forms a subgyrogroup of $H^\bullet$ that is isomorphic to $H$ as a gyrogroup.

Let $e$ be the identity of $H$, and denote by $E$ the set of all step functions $f$ from $J$ to $H$ satisfying the following condition:
\begin{itemize}
\item[(i)] there exists $b \in [0,1)$ such that $f(x) = e$ for each $x$ with $b \leq x < 1$.
\end{itemize}
It is clear that $E$ is a subgyrogroup of $H^\bullet$. Let $D=\{g_{k}: k\in\omega\}$ be a countable dense subgyrogroup of $H$.
Denote by $E'$ the subgyrogroup of $E$ consisting of all $f \in E$ satisfying the following condition:



\begin{itemize}
\item[(ii)] there exists rational numbers $0 = b_0 < b_1 < \dots < b_{m-1} < b_m =1$ such that $f$ is constant on each subinterval $[b_k, b_{k+1})$ and $f(b_k) = g_k \in D$ for $k = 0, 1, \dots, m-1$.
\end{itemize}
Notice that $E'$ is countable.
The argument given in the proof of \cite[Theorem 3.11]{8} shows that $E'$ is dense in $H^\bullet$.

Denote by $G_0$ the subgyrogroup of $H^\bullet$ generated by the set $i(G) \cup E$.
Since $i: H \rightarrow H^\bullet$ is a topological gyrogroup monomorphism, the gyrogroup $G$ is topologically isomorphic to the subgyrogroup $i(G)$ of $H^\bullet$.
From $E' \subseteq E \subseteq G_0$, it follows that the gyrogroup $G_0$ is separable. 
Let us verify that $i(G)$ is closed in $G_0$.

First we note that $i(H)$ is closed in $H^\bullet$ according to \cite[Theorem 3.3]{8}.
Hence the required conclusion about $i(G)$ will follow if we show that $G_0 \cap i(H) = i(G)$.
Assume that $f \in G_0 \cap i(H)$.
Then $f$ is a constant function on $J$ with a single value $h_0 \in H$.
We have to show that $h_0 \in G$, that is, $f \in i(G)$.
As $f \in G_0$, the function $f$ must necessarily be composed of a finite number of elements $\{g_1^\bullet, g_2^\bullet, \dots, g_{k+1}^\bullet \}$ from $i(G)$ and a finite number of elements $\{t_1, t_2, \dots, t_k\}$ from $E$, combined through a finite number of gyroadditions ``$\oplus$'' and inversions ``$\ominus$'' (with a specific parenthesization), we can write $f$ in the form $$f = W(g_1^\bullet, g_2^\bullet, \dots, g_{k+1}^\bullet, t_1, \dots, t_k).$$
Item (i) of our definition of the gyrogroup $E$ implies that there exists $b < 1$ such that $t_i(b) = e$ for each $i = 1, \dots, k$.
Hence $h_0 = f(b) = W(g_1^\bullet(b), g_2^\bullet(b), \dots, g_{k+1}^\bullet(b)) = W(g_1, \dots, g_k, g_{k+1}) \in G$.
Since $f$ is a constant function, we see that $f \in i(G)$.
This implies the inclusion $G_0 \cap i(H) \subseteq i(G)$. The inverse inclusion is obvious. Therefore, $G \cong i(G)$ is closed in $G_0$.

Now we have to check that the gyrogroup $G_0$ is path-connected and locally path-connected. It is worth mentioning that $G_0$ is a proper dense subgyrogroup of $H^\bullet$, 
but not every dense subgyrogroup of $H^\bullet$ inherits these properties from $H^\bullet$. We start with the path-connectedness of $G_0$.

Since $G_0$ is generated by the set $i(G) \cup E$, it suffices to verify that, for every element $f\in i(G) \cup E$, there exists a path in $G_0$ connecting the identity $e^\bullet$ of $H^\bullet$ with $f$.
Indeed, every element $f \in G_0$ is a product of finitely many elements $f_1, \dots, f_n$ of $i(G) \cup E$ and, multiplying the paths connecting $f_1, \dots, f_n$ with $e^\bullet$, we obtain a path in $G_0$ connecting $e^\bullet$ and $f$. Now take an arbitrary element $f \in i(G) \cup E$.

{\bf Case 1:} $f \in i(G)$.

Then $f = g^\bullet$ for some $g \in G$.
For every $r \in [0,1]$, let $f_r$ be a step function from $J$ to $H$ defined by $f_r(x) = g$ if $x < r$ and $f_r(x) = e$ if $r \leq x < 1$.
It is clear that $f_r \in H^\bullet$ for each $r \in [0,1]$. We claim that the mapping $\varphi: [0,1] \rightarrow H^\bullet$, $\varphi(r) = f_r$ ($r\in [0, 1]$), is continuous.
Indeed, it suffices to prove that for any open neighborhood of $f_{r_0} \oplus^\bullet O(V, \varepsilon)$ of $f_{r_0}$ in $H^\bullet$, there exists a $\delta > 0$ such that $|r-r_0| < \delta$ implies $\varphi(r) = f_r \in f_{r_0} \oplus^\bullet O(V, \varepsilon)$, where $V$ is an open neighborhood of the identity element in $H$ and $\varepsilon>0$.
Equivalently, we need to demonstrate that $\ominus^\bullet f_{r_0} \oplus^\bullet f_r \in O(V, \varepsilon)$.
From the definitions of $f_{r_0}, f_r$, we have $\mu(\{x \in J: \ominus^\bullet f_{r_0} (x)\oplus^\bullet f_r(x) \notin V \}) \leq |r_0 - r|$.
Then we can choose $\delta = \varepsilon$.
Whenever $|r - r_0| \leq \delta = \varepsilon$, then $\varphi(r) = f_r \in f_{r_0} \oplus^\bullet O(V, \varepsilon)$.
Therefore, $\varphi$ is continuous. Since $\varphi(0) = e^\bullet$, $\varphi(1) = f \in i(G)$, and $f_r \in E \subseteq G_0$ for each $r \in [0,1)$,
this proves that $\varphi$ is a path in $G_0$ connecting $e^\bullet$ and $f$.

{\bf Case 2:} $f \in E$.

Choose a partition $0 = b_0 < b_1 < \dots < b_m = 1$ of $J$ such that $f$ is constant on each interval $[b_i, b_{i+1})$, where $0 \leq i < m$, and $f(b_{m-1}) = e$.
Let us define a path $\varphi: [0,1] \rightarrow G_0$ connecting $e^\bullet$ with $f$ as follows.
First we put $l_k = b_{k+1} - b_k$ for $k = 0, \dots, m-1$. For every $r \in [0,1]$ and every $x \in J$, let
$$f_r(x) =
\begin{cases}
f(x), & \text{if } b_k \le x < b_k + r \cdot l_k \text{ for some } k \text{ with } 0 \le k < m; \\
e, & \text{if } b_k + r \cdot l_k \le x < b_{k+1} \text{ for some } k \text{ with } 0 \le k < m.
\end{cases}
$$
It is easy to verify that $f_0 = e^\bullet$, $f_1 = f$, and $f_r(x) = e$ if $b_{m-1} \leq x < 1$. Hence $f_r \in E \subseteq G_0$ for each $r \in [0,1]$.
Again, the mapping $\varphi: [0,1] \rightarrow H^\bullet$ defined by $\varphi(r) = f_r$ for each $r \in [0,1]$ is continuous, so $\varphi$ is a path in $G_0$ connecting $e^\bullet$ and $f$.
To prove the mapping $\varphi$ is continuous, it suffices to prove that for any open neighbourhood $f_{r_0} \oplus^\bullet O(V, \varepsilon)$ of $f_{r_0}$ in $H^\bullet$, there exists a $\delta > 0$ such that $|r-r_0| < \delta$ implies $\varphi(r) = f_r \in f_{r_0} \oplus^\bullet O(V, \varepsilon)$, where $V$ is an open neighborhood of the identity element in $H$ and $\varepsilon>0$.
Equivalently, we need to demonstrate that $\ominus^\bullet f_{r_0} \oplus^\bullet f_r \in O(V, \varepsilon)$.
From the definitions of $f_{r_0}, f_r$, we have $$\mu(\{x \in J: \ominus^\bullet f_{r_0}(x) \oplus^\bullet f_r(x) \notin V \}) \leq \sum_{k=0}^{m-1}(|r_0 - r| \cdot l_k) = |r_0 - r|\sum_{k=0}^{m-1}l_k.$$ Since $\sum_{k=0}^{m-1}l_k = 1,$ it follows that $\mu(\{x \in J: \ominus^\bullet f_{r_0}(x) \oplus^\bullet f_r(x) \notin V \}) \leq |r_0 - r|$.
Then we can choose $\delta = \varepsilon$.
Whenever $|r - r_0| \leq \delta = \varepsilon$, we have $\varphi(r) = f_r \in f_{r_0} \oplus^\bullet O(V, \varepsilon)$.
Therefore, $\varphi$ is continuous.

Summing up, the gyrogroup $G_0$ is path-connected.

Finally, we check that $G_0$ is locally path-connected.
Every neighborhood of $e^\bullet$ in $H^\bullet$ contains an open neighborhood of the form $$O(U, \varepsilon) = \{f \in H^\bullet: \mu(\{x \in J: f(x) \notin U\}) < \varepsilon \},$$ where $U$ is an open neighborhood of the identity $e$ in $H$ and $\varepsilon > 0$. Therefore, by the homogeneity of $G_0$, it suffices to verify that each intersection $G_0 \cap O(U,\varepsilon)$ is path-connected. Now take an arbitrary element $f \in G_0 \cap O(U,\varepsilon)$.
Then $f = W(i(g_1), \dots, i(g_k), i(g_{k+1}), t_1, \dots, t_k)$, where $g_1, \dots, g_k, g_{k+1} \in G$, $t_1, \dots, t_k \in E$.
Our aim is to define a path $\Phi:[0,1] \rightarrow G_0 \cap O(U,\varepsilon)$ connecting $e^\bullet$ with $f$.

First we choose a partition $0 = b_0 < b_1 < \dots < b_{m-1} < b_m =1$ of $J$ such that $t_i$ is constant on $[b_j,b_{j+1})$ for all integers $i \leq k$ and $j < m$.
Let also $l_j = b_{j+1} - b_j$, where $j = 0, \dots, m-1$. For every $i =1, \dots, k, k+1$ we define a path $\varphi_{i} : [0,1] \rightarrow H^\bullet$ by
$$\varphi_i(r,x) =
\begin{cases}
g_i, & \text{if } b_j \le x <b_j + r \cdot l_j \text{ for some } j \text{ with } 0 \le j < m; \\
e, & \text{if } b_j + r \cdot l_j \le x < b_{j+1} \text{ for some } j \text{ with }  0 \le j < m. \\
\end{cases}
$$
Then $\varphi_i(0,x) = e$, $\varphi_i(1,x) =g_i$ for each $x \in J$, and $\varphi_i(r,\cdot) \in G_0$ for each $r \in [0,1]$.
The path $\varphi_i$ is continuous and connects $e^\bullet$ with $g_i^\bullet$ in $G_0$.

Similarly, we define a path $\psi_i: [0,1] \rightarrow H^\bullet$ for each $i = 1, \dots, k$ by
$$\psi_i(r, x) =
\begin{cases}
t_i(x), & \text{ if } b_j \leq x < b_j + r \cdot l_j \text{ for some } j \text{ with } 0 \leq j <m ;\\
e, & \text{ if } b_j +r \cdot l_j \leq x <b_{j+i} \text{ for some } j \text{ with } 0 \leq j < m.\\
\end{cases}
$$
It is clear that $\psi_i(0, x) = e$, $\psi_i(1, x) = t_i(x)$ for each $x \in J$ and $\psi_i(r, \cdot) \in G_0$ for each $r \in [0, 1]$. The path $\psi_i$ is continuous and connects $e^\bullet$ with $t_i$ in $G_0$.


Finally, we define a path $\Phi$ in $G_0$ connecting $e^\bullet$ with $f$ by letting
$$\Phi(r,x) = W(\varphi_1(r,x), \dots, \varphi_k(r,x), \varphi_{k+1}(r,x), \psi_1(r,x), \dots, \psi_k(r,x)),$$
where $r \in [0,1]$ and $x \in J$.
The path $\Phi$ is continuous being a product of continuous pathes $\varphi_i$ ($i\leq k+1$) and $\psi_j$ ($j\leq k$).
The following Claim describes a basic property of the path $\Phi$:

{\bf Claim:} For all $r \in [0,1]$ and $x \in J$, either $\Phi(r,x) = f(x)$ or $\Phi(r,x) = e$.

Indeed, let $r \in [0,1]$ and $x \in J$ be arbitrary. Choose an integer $j < m$ such that $b_j \leq x < b_{j+1}$.
If $b_j \leq x < b_j +rl_j$, then $\varphi_i(r,x) = g_i$ and $\psi_i(r,x) = t_i(x)$ for all $i$, whence it follows that $\Phi(r,x) = f(x)$.
If $b_j + r \cdot l_j \leq x <b_{j+1}$, then $\varphi_i(r,x) = e$ and $\psi_i(r,x) = e$ for all $i$, so $\Phi(r,x) = e$. This proves our Claim.

Applying Claim we see that
$$\{x \in J: \Phi(r,x) \notin U \} \subseteq \{x \in J: f(x) \notin U \},$$
for every $r \in[0,1]$.
Hence $\mu(\{x \in J: \Phi(r,x) \notin U \}) < \varepsilon$ for each $r \in [0,1]$.
In other words, the path $\Phi$ lies in $O(U, \varepsilon)$, so the set $O(U, \varepsilon)$ is path-connected.
Since the sets of the form $G_0 \cap O(U, \varepsilon)$ constitute a base for $G_0$ at the identity, this completes the proof of this theorem.
\end{proof}

\end{document}